\documentclass[a4paper,11pt]{article}
\usepackage[english]{babel}
\usepackage[latin1]{inputenc}
\usepackage{lmodern}
\usepackage{amsfonts}
\usepackage{bm}
\usepackage{amssymb}
\usepackage{latexsym}
\usepackage{amsmath}
\usepackage{amscd}
\usepackage{amsthm}
\usepackage{pifont}
\usepackage{fancyhdr}
\usepackage{epsfig}
\usepackage{alltt}
\usepackage{hyperref}
\usepackage{multimedia}
\usepackage[T1]{fontenc}
\usepackage{psfrag}
\usepackage{url}
\usepackage{graphicx}
\usepackage{psfrag}
\usepackage{ulem}

\newtheorem{theorem}{Theorem}

\newtheorem{lemme}[theorem]{Lemma}

\def\card{\mathrm{Card}}

\newcommand{\EE}{\mathop{\hbox{\rm I\kern-0.17em E}}\nolimits}
\newcommand{\PP}{\mathop{\hbox{\rm I\kern-0.17em P}}\nolimits}

\newcommand{\E}{{\mathbb E}}
\newcommand{\N}{{\mathbb N}}
\newcommand{\Z}{{\mathbb Z}}
\renewcommand{\P}{{\mathbb P}}

\newcommand{\R}{{\mathbb R}}

\newcommand{\ind}{{\bf 1}}

\def\i{\mathbf{i}}

\def\ie{\textit{i.e.} }

\def\eg{\textit{e.g.} }
\def\etal{\textit{et al.} }

\title{The 2d-Directed Spanning Forest is almost surely a tree}

\author{David Coupier, Viet Chi Tran}

\date{\today}

\begin{document}
\maketitle

\begin{abstract}
We consider the Directed Spanning Forest (DSF) constructed as follows: given a Poisson point process N on the plane, the ancestor of each point is the nearest vertex of N having a strictly larger abscissa. We prove that the DSF is actually a tree. Contrary to other directed forests of the literature, no Markovian process can be introduced to study the paths in our DSF. Our proof is based on a comparison argument between surface and perimeter from percolation theory. We then show that this result still holds when the points of N belonging to an auxiliary Boolean model are removed. Using these results, we prove that there is no bi-infinite paths in the DSF.
\hfill $\Box$
\end{abstract}

\noindent Keywords: Stochastic geometry; Directed Spanning Forest; Percolation.\\
\noindent AMS Classification: 60D05

\section{Introduction}

\indent
Let us consider a homogeneous Poisson point process $N$ (PPP) on $\R^2$ with intensity $1$. The plane $\R^2$ is equipped with its canonical orthonormal basis $(O,e_{x},e_{y})$ where $O$ denotes the origin $(0,0)$. In the sequel, the $e_{x}$ and $e_{y}$ coordinates of any given point of $\R^2$ are respectively called its abscissa and its ordinate, and often denoted by $(x,y)$.\\
From the PPP $N$, Baccelli and Bordenave \cite{baccellibordenave} defined the \textit{Directed Spanning Forest} (DSF) with direction $e_x$ as a random graph whose vertex set is $N$ and whose edge set satisfies: the ancestor of a point $X\in N$ is the nearest point of $N$ having a strictly larger abscissa. In their paper, the DSF appears as an essential tool for the asymptotic analysis of the Radial Spanning Tree (RST). Indeed, the DSF can be seen as the limit of the RST far away from its root.\\
Each vertex $X$ of the DSF almost surely (a.s.) has a unique ancestor (but may have several children). So, the DSF can have no loop. This is a forest, \ie a union of one or more disjoint trees. The most natural question one might ask about the DSF is whether the DSF is a tree. The answer is ``yes'' with probability $1$.

\begin{theorem}
\label{th1}
The DSF constructed on the homogeneous PPP $N$ is a.s. a tree.
\end{theorem}

\noindent
Let us remark that the isotropy and the scale-invariance of the PPP $N$ imply that Theorem \ref{th1} still holds when the direction $e_{x}$ is replaced with any given $u\in\R^2$ and for any given value of the intensity of $N$.\\
Theorem \ref{th1} means that a.s. the paths in the DSF, with direction $e_{x}$ and coming from any two points $X,Y\in N$, eventually coalesce. In other words, any two points $X,Y\in N$ have a common ancestor somewhere in the DSF. This can be seen on simulations of Figure \ref{fig1}.

\begin{figure}[!ht]
\begin{center}
\begin{tabular}{cc}
(a) & (b)\\
\includegraphics[width=6cm,height=6cm]{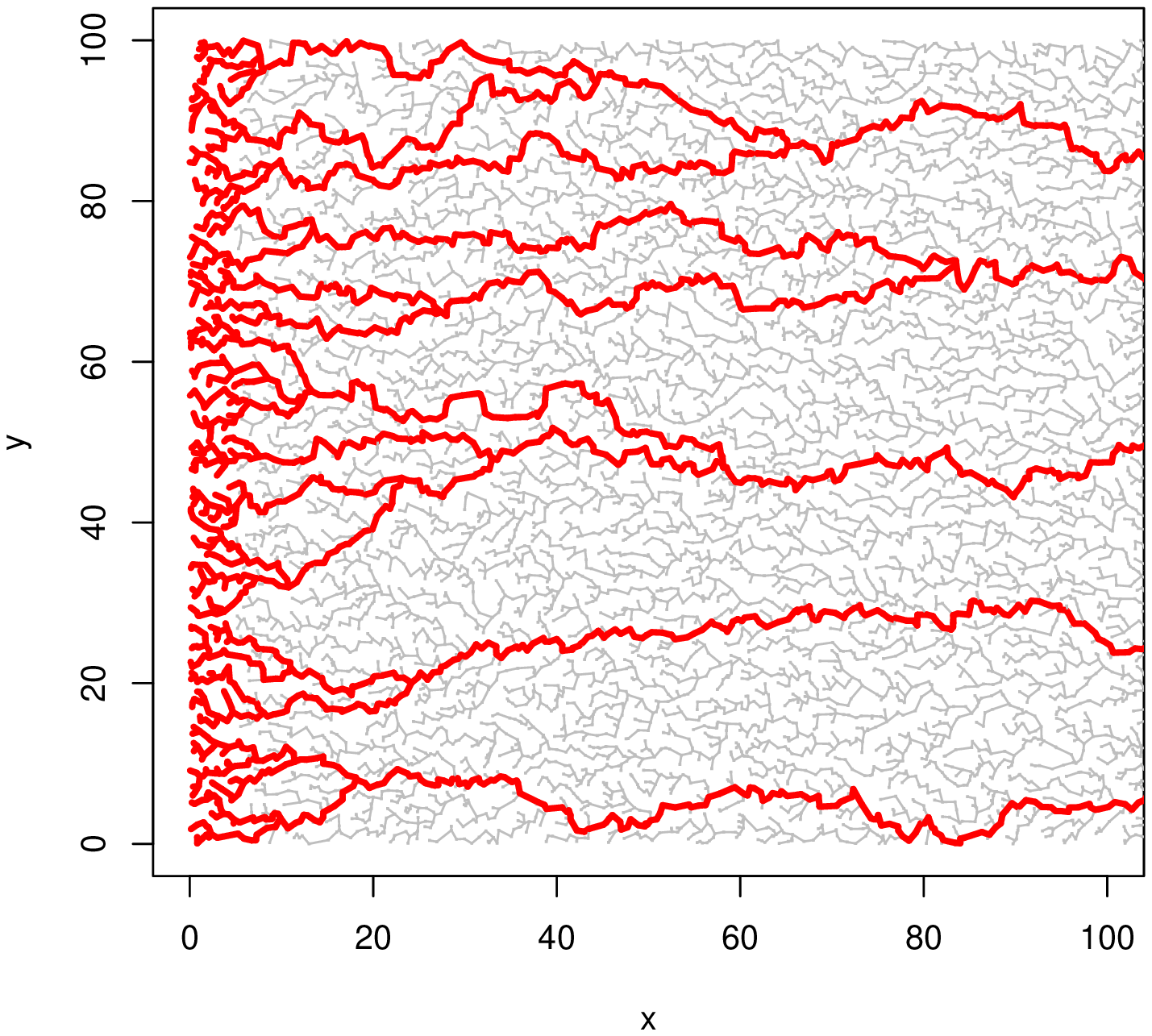}
 & \hspace{-0.5cm}\includegraphics[width=10cm,height=6cm]{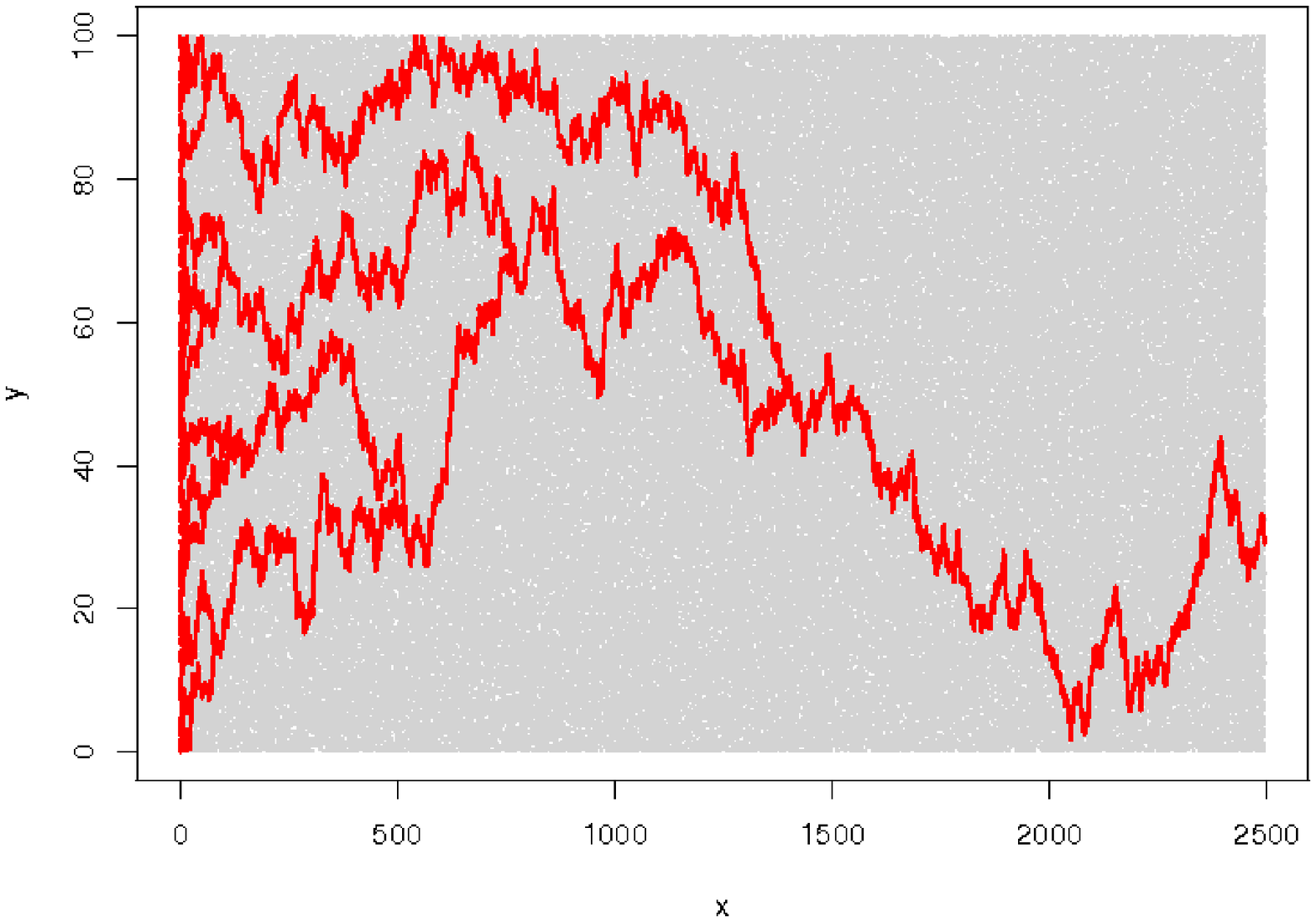}
 \end{tabular}
 \vspace{-0.5cm}
\caption{{\small \textit{Simulations of the Directed Spanning Forest. The paths with direction $e_{x}$ and coming from vertices with abscissa $0\leq x\leq 5$ and ordinates $0\leq y\leq 100$ and their ancestors are represented in bold red lines.}}}\label{fig1}
\end{center}
\end{figure}

\noindent
In the context of first and last passage percolation models on $\Z^2$, Licea and Newman \cite{liceanewman} and Ferrari and Pimentel \cite{FP} have stated that infinite paths with the same direction have to coalesce. For technical reasons, the coalescence is harder to prove in continuum models, typically when the lattice $\Z^2$ is replaced by a PPP on $\R^{2}$. This has been done by Howard and Newman \cite{howardnewman2} for (continuum) first passage percolation models. In the same spirit, Alexander \cite{alexander} studies the number of topological ends (i.e. infinite self-avoiding paths from any fixed vertex) of trees contained in some minimal spanning forests.\\
The DSF is radically different from the graphs described in the papers mentioned previously. Indeed, the construction of paths in the DSF only requires local information whereas that of first and last passage paths need to know the whole PPP.\\

Similar directed forests built on local knowledge of the set of vertices have been studied. Consider the two dimensional lattice $\Z^{2}$ where each vertex is open or closed according to a site percolation process. Gangopadhyay \etal \cite{gangopadhyayroysarkar} connect each open vertex $(x,y)$ to the closest open vertex $(x',y')$ such that $x'=x+1$ (with an additional rule to ensure uniqueness of the ancestor $(x',y')$). Athreya \etal \cite{athreyaroysarkar} choose the closest open vertex $(x',y')$ in the $\frac{\pi}{2}$ lattice cone generated at $(x,y)$ and with direction $e_{x}$. Ferrari \etal \cite{ferrarilandimthorisson} connect each point $(x,y)$ of a PPP on $\R^2$ to the first point $(x',y')$ of the PPP whose coordinates satisfy $x'>x$ and $|y-y'|<1$. In these three works, it is established that the obtained graph is a.s. a tree. \\
However, these three models offer a great advantage; the choice of the ancestor of the vertex $(x,y)$ does not depend on what happens before abscissa $x$. This crucial remark allows to introduce easily Markov processes (indexed by the abscissa), the use of martingale convergence theorems and of Lyapunov functions. This is no longer the case in the DSF considered here. Indeed, given a descendant $Y$ of $(x,y)$ and its ancestor $Y'$, the ball $B(Y,|Y-Y'|)$ may overlap the half-plane $\{x'>x\}$ and the ancestor of $(x,y)$ cannot be in the resulting intersection. See Figure \ref{fig:notmarkov} for an illustration of this phenomenon. Because of this, Bonichon and Marckert \cite{bonichonmarckert}, who consider navigation on random set of points in the plane, have to restrict to cones of width at most $2\pi/3$ for instance and their techniques do not apply to our case with $\pi$. \\
Notice also our discrete model bears similarities with the Brownian web (\eg \cite{fontesisopinewmanravishankar}), for which it is known that the resulting graph is a tree with no bi-infinite path. See Figure \ref{fig1} (b).

\begin{figure}[!ht]
\begin{center}
\psfrag{y}{\small{$Y$}}
\psfrag{y'}{\small{$Y'$}}
\psfrag{x}{\small{$(x,y)$}}
\includegraphics[width=2.1cm,height=5cm]{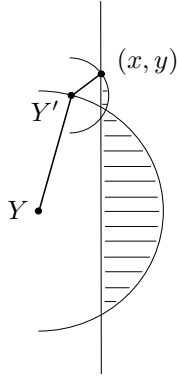}
\caption{{\small \textit{$Y$, $Y'$ and $(x,y)$ are points of the PPP $N$. In the DSF with direction $e_{x}$, $Y$ and $Y'$ are some descendants of $(x,y)$ and $Y'$ is the ancestor of $Y$. The ancestor of $(x,y)$ cannot belong to the hatched region.}}}\label{fig:notmarkov}
\end{center}
\end{figure}

\noindent
Our strategy to prove Theorem \ref{th1} is inspired from the percolation literature. It is based on a comparaison argument between surface and perimeter due to Burton and Keane \cite{burtonkeane}, and used in \cite{FP,howardnewman2,liceanewman}. The key properties of the point process $N$ that are needed are its invariance by a group of translations and the independence between the restrictions of $N$ to disjoint sets. This independence is to be used carefully though: for instance, modifying the PPP locally in a bounded region may have huge consequences on the paths of the DSF started in another region.\\
Then, Theorem \ref{th1} is extended to random environment. More precisely, if we delete all the points of the PPP $N$ which belong to an auxiliary (and independent) Boolean model, then the DSF constructed on the remaining PPP is still a tree (Theorem \ref{th1boolean}).\\
Finally, Theorem \ref{thbiinf} states that, in the DSF with direction $e_{x}$ and with probability one, the paths with direction $-e_{x}$ coming from any vertex $X$ are all finite. Combining this with Theorem \ref{th1}, this means that the DSF has a.s. one topological end.\\

\indent
In Section \ref{proofpart1}, a key result (Lemma \ref{corol2}) is stated from which stems the comparison argument between surface and perimeter that allows to prove Theorem \ref{th1}. Section \ref{section:P(F)>0} is devoted to the proof of Lemma \ref{corol2}. In Section \ref{sectiontrous}, we show that Theorem \ref{th1} still holds when the set on which $N$ is defined becomes a random subset of $\R^{2}$. Finally, in Section \ref{sectionbiinfini}, we use the fact that the DSF is a tree to prove, as an illustration, that there is no bi-infinite paths in the DSF.

\section{Proof of Theorem \ref{th1}}
\label{proof}

\subsection{Any two paths eventually coalesce}
\label{proofpart1}

For $X\in N$, let us denote by $\gamma_X$ the path of the DSF started at $X$ and with direction $e_x$. It is composed of the ancestors of $X$, with edges between consecutive ancestors. This path will be identified with the subset of $\R^2$ corresponding to the union of segments $[X';X'']$ in $\R^2$ where $X'$ and $X''$ are two consecutive ancestors of $X$. By construction, this path is infinite.\\
Let $X,Y\in N$. Either the paths $\gamma_X$ and $\gamma_Y$ coalesce in a point $Z\in N$ which will be the first common ancestor to $X$ and $Y$, and coincide beyond this point. Or they do not cross and are disjoint subsets of $\R^2$. This trivial remark will be widely used in Section \ref{section:P(F)>0}.\\
Let us denote by $\mathcal{N}\in \{1,2,\dots,+\infty\}$ the number of \textit{disjoint} infinite paths of the DSF. Theorem \ref{th1} is obviously equivalent to $\P(\mathcal{N}=1)=1$. Now, we assume that
\begin{equation}
\P(\mathcal{N}\geq 2)>0
\end{equation}
and our purpose is to obtain a contradiction. We first state the key result (Lemma \ref{corol2}) which will be proved in Section \ref{section:P(F)>0} and then explain how it leads to a contradiction.\\

Before starting the proof, let us remark that the ergodicity of the PPP $N$ implies that the variable $\mathcal{N}$ is constant almost surely. So without loss of generality, we could assume $\P(\mathcal{N}\geq 2)=1$ instead of $\P(\mathcal{N}\geq 2)>0$. See \cite{MR}, Chapter 2.1 for details on ergodicity applied to stationary point proceses.\\

Let $\N^*=\{1,2,\dots \}$ be the set of positive integers. For any $m,M\in \N^*$, let us denote by $C_{m,M}$ the cell $[-m,m)\times[-M,M)$. Let $F_{m,M}$ be the following event: there exists a path $\gamma_X$ in the DSF with $X\in C_{m,M}$ that does not meet any other path $\gamma_Y$ for all $Y\in \{x < m\}\setminus C_{m,M}$:
\begin{equation}
F_{m,M} = \big\{ \exists X \in N \cap C_{m,M} , \; \forall Y \in N \cap \{x < m\} \setminus C_{m,M} , \; \gamma_X \cap \gamma_Y = \emptyset \big\} ~.
\label{defF_mM}
\end{equation}

\smallskip
\begin{lemme}
\label{corol2}
If $\P(\mathcal{N}\geq 2)>0$ then there exist some positive integers $m,M$ such that $\P\big(F_{m,M}\big)>0$.
\end{lemme}

\medskip

Let $L\in \N^*$. Let us consider the lattice $\mathcal{Z}_{L,m,M}$ consisting of the $(2L+1)^2$ points $z=(2mk,2M\ell)$ where $-L\leq k,\ell\leq L$, and the union $\mathcal{R}_{L,m,M}$ of nonoverlapping translated cells $C^{z}_{m,M}=z+C_{m,M}$ indexed by $z\in\mathcal{Z}_{L,m,M}$. We also define $F^z_{m,M}$ as the translated event of $F^{O}_{m,M}=F_{m,M}$ to the cell $C^{z}_{m,M}$.\\
Thanks to the translation invariance in distribution of the PPP $N$, all the $F^z_{m,M}$'s, for $z\in \mathcal{Z}_{L,m,M}$, have the same probability. Furthermore, if both events $F^z_{m,M}$ and $F^{z'}_{m,M}$ occur with $z\not=z'$ then there are in the DSF two disjoint infinite paths started respectively in $C^{z}_{m,M}$ and $C^{z'}_{m,M}$. Hence, the number of events $F^z_{m,M}$'s occurring simultaneously is smaller than the number $\eta_{L,m,M}$ of edges of the DSF going out of the rectangle $\mathcal{R}_{L,m,M}$. It follows that for any $L,m,M$:
$$
\E \left\lbrack \eta_{L,m,M} \right\rbrack \geq \E \left\lbrack \sum_{z\in\mathcal{Z}_{L,m,M}} \ind_{F^z_{m,M}} \right\rbrack = (2L+1)^2 \P \big( F_{m,M} \big) ~.
$$
Now, using the integers $m,M$ given by Lemma \ref{corol2}, the probability $\P\big(F_{m,M}\big)$ is positive. We deduce that the expectation $\E[\eta_{L,m,M}]$ grows at least as $L^2$. The contradiction comes from the next result (Lemma \ref{lemme:eta}) in which it is proved that $\E[\eta_{L,m,M}]$ is at most of order $L^{3/2}$. This results from the fact that the expected number of edges crossing the boundary of $\mathcal{R}_{L,m,M}$ should be of an order close to the perimeter of this rectangle. This concludes the proof of Theorem \ref{th1}.

\begin{lemme}
\label{lemme:eta}
For all $m,M\in \N^*$ there exists a constant $C>0$ (depending on $m,M$) such that for all $L\in \N^*$,
$$
\E [ \eta_{L,m,M} ] \leq C L^{3/2} ~.
$$
\end{lemme}

\begin{proof}
Let us write the random variable $\eta_{L,m,M}$ as the sum
$$
\eta_{L,m,M}=\eta^{<}_{L} + \eta^{>}_{L}
$$
where $\eta^{<}_{L}$ and $\eta^{>}_{L}$ respectively denote the number of edges exiting the rectangle $\mathcal{R}_{L,m,M}$ and with lengths shorter and longer than $\sqrt{L}$. It suffices to prove that their expectations are of order $L^{3/2}$.\\
Each of the $\eta^{<}_{L}$ edges exiting $\mathcal{R}_{L,m,M}$ and shorter than $\sqrt{L}$ has an extremity in a strip $S_{L}$ of width $\sqrt{L}$ all around the rectangle $\mathcal{R}_{L,m,M}$. This means $\eta^{<}_{L}$ is smaller than the number of points of the PPP $N$ in $S_{L}$. Its expectation is upper bounded by the area of the strip $S_{L}$ which is of order $L^{3/2}$.\\
The number $\eta^{>}_{L}$ of edges exiting the rectangle $\mathcal{R}_{L,m,M}$ and longer than $\sqrt{L}$ is necessarily smaller than the number $\xi$ of points of $N\cap\mathcal{R}_{L,m,M}$ for which the distance to their ancestor is larger than $\sqrt{L}$. To each of these points $X$ is associated an open half disc centered at $X$ and with radius $\sqrt{L}$ in which there is no point of $N$. Otherwise, the ancestor of $X$ would be at a distance smaller than $\sqrt{L}$. To sum up, the points counted by $\xi$ all belong to the rectangle $\mathcal{R}_{L,m,M}$ and are at distance from each other larger than $\sqrt{L}$. Their number can not exceed the order $L$. So do $\E[\xi]$ and $\E[\eta^{>}_{L}]$. This proves the announced result.
\end{proof}

\subsection{Proof of the key lemma}
\label{section:P(F)>0}

This section is devoted to the proof of Lemma \ref{corol2}, \ie to prove that the event $F_{m,M}$ defined in (\ref{defF_mM})
occurs with positive probability under the hypothesis $\P(\mathcal{N}\geq 2)>0$. The proof can be divided into two steps. First, we prove (Lemma \ref{3paths}) that there exist positive integers $m,M$ such that with positive probability three disjoint infinite paths come from the cell $C_{m,M}$. Hence, the intermediate one, say $\gamma$, is trapped between the two other paths on the half-plane $\{x \geq m\}$. Then, a local modification of this event prevents any path $\gamma_Y$ with $Y\in \{x < m\}\setminus C_{m,M}$ to touch $\gamma$. Roughly speaking, we build a shield protecting $\gamma$ from other paths $\gamma_Y$'s. See Figure \ref{fig:shield}.\\

To show that with positive probability three disjoint infinite paths come from a cell $C_{m,M}$, we start with two disjoint paths (Lemma \ref{2paths}). For any positive integers $m,M$, let us denote by $E_{m,M}$ the east side of the cell $C_{m,M}$
\begin{equation}
E_{m,M} = \{ (x,y) \,;\, x = m \; \mbox{and} -M \leq y < M \}~.
\label{def:EmM}\end{equation}
Let $\delta>0$ and $m,\,M\in \N^*$ such that $m>\delta$. We define the event
\begin{equation*}
 A^\delta_{m,M}=\left\{\exists X,Y\in N\cap C_{m-\delta,M}\; \mbox{s.t.}\; \gamma_X\cap \gamma_Y=\emptyset \mbox{ and } (\gamma_X\cup \gamma_Y)\cap \partial C_{m,M} \subset E_{m,M} \right\}~,
\end{equation*}where $\partial C_{m,M}$ denotes the boundary of $C_{m,M}$. The event $A^\delta_{m,M}$ says that there exist two disjoint paths $\gamma_X$ and $\gamma_Y$ coming from points $X,Y\in N\cap C_{m-\delta,M}$ and exiting $C_{m,M}$ by its east side $E_{m,M}$. For technical reasons in the proof, we will require the parameter $\delta$ so that the two disjoint paths exit $C_{m,M}$ by crossing the rectangle delimited by the two segments $E_{m-\delta,M}$ and $E_{m,M}$.

\begin{lemme}
\label{2paths}If $\P(\mathcal{N}\geq 2)>0$ then, for all $\delta>0$, there exist $m,M\in \N^*$ with $m>\delta$ such that $\P\big(A^\delta_{m,M}\big)>0$.
\end{lemme}

\begin{proof}
The inequalities
\begin{eqnarray*}
0< \P( \mathcal{N}\geq 2) & = & \P \Big( \exists m,M \in \N^* , \; \exists X,Y \in N \cap C_{m,M}, \; \gamma_X \cap \gamma_Y = \emptyset \Big) \\
& \leq & \sum_{m,M \in \N^*} \P \Big( \exists X,Y \in N \cap C_{m,M} , \; \gamma_X \cap \gamma_Y = \emptyset \Big)
\end{eqnarray*}
imply that at least one of the terms of the above sum is positive. Let $m,M\in\N^*$ be the corresponding deterministic indices.

Let $X_0=X,X_1,X_2,\dots$ be the sequence of successive ancestors of $X\in N$ in the DSF with respect to direction $e_x$. Theorem 4.6 of \cite{baccellibordenave} puts forward the existence of an increasing sequence of finite stopping times $(\theta_k)_{k\geq 0}$ such that $\theta_0=0$ and the vectors $(X_{\theta_k}-X_{\theta_{k}+1},\dots, X_{\theta_{k+1}-1}-X_{\theta_{k+1}})_{k>0}$ are i.i.d. Consequently, $\gamma_X=\{X,X_1,X_2,\dots\}$ a.s. eventually goes out of the half-plane $\{x < m+\delta\}$. So, 
\begin{multline}
0< \P \Big( \exists X,Y \in N \cap C_{m,M} \; \mbox{ s.t.} \; \gamma_X \cap \gamma_Y = \emptyset \Big) \\
= \lim_{R\to+\infty} \P \left(
\begin{array}{c}
\exists X,Y \in N \cap C_{m,M} \; \mbox{ s.t.} \; \gamma_X \cap \gamma_Y = \emptyset \\
\mbox{and} \; (\gamma_X \cup \gamma_Y) \cap \partial C_{m+\delta,R}\subset E_{m+\delta,R}
\end{array}
\right) ~.\label{etape4}
\end{multline}Thus, there exists an integer $R$ large enough so that the probability in the right hand side of \eqref{etape4} is positive. Replacing $M$ by $\max(R,M)$ and $m$ by $m-\delta$ provides the announced result.
\end{proof}

We are now able to state a result similar to Lemma \ref{2paths}, but with three paths instead of two. Let us introduce, for $\delta>0$ and $m,M\in \N^*$ such that $m>\delta$, the event $B_{m,M}^\delta$ defined as follows: there exist three disjoint paths $\gamma_X$, $\gamma_Y$ and $\gamma_Z$ coming from points $X,Y,Z\in N\cap C_{m-\delta,M}$ and exiting the cell $C_{m,M}$ by its east side $E_{m,M}$.

\begin{lemme}
\label{3paths}
If $\P(\mathcal{N}\geq 2)>0$ then, for all $\delta>0$, there exist $m,M\in \N^*$ with $m>\delta$ such that $\P\big(B^\delta_{m,M}\big)>0$.
\end{lemme}

Let $\gamma_X$, $\gamma_Y$ and $\gamma_Z$ be the disjoint paths given by $B^\delta_{m,M}$. The parameter $\delta$ ensures that one of them, say $\gamma_X$, is trapped between $\gamma_Y$ and $\gamma_Z$ on $\{x\geq m-\delta\}$ and not only on $\{x\geq m\}$. This precaution is needed to get the inclusion (\ref{inclusionfinale}) in the sequel. Remark also, $\delta=1$ will be sufficient to prove Theorem \ref{th1}. Lemma \ref{3paths} (with any positive $\delta$) will be used to prove Theorem \ref{th1boolean}, in the next section.

\begin{proof}
Let $m,M$ be the positive integers given by Lemma \ref{2paths}. First, for any integer $\ell\in\Z$, let us consider the translated event $A_{m,M}^{\delta,\ell}$ of $A_{m,M}^{\delta,0}=A_{m,M}^\delta$ to the cell $(0,2M\ell)+C_{m,M}$, \ie there exist two disjoint paths $\gamma_X$ and $\gamma_Y$ coming from points $X,Y\in N\cap ((0,2M\ell)+C_{m-\delta,M})$ and exiting $(0,2M\ell)+C_{m,M}$ by its east side. By stationarity, all these events have the same probability. Moreover, if we assume that for all $\ell,k\in\Z$, the probabilities $\P\big(A_{m,M}^{\delta,\ell}\cap A_{m,M}^{\delta,k}\big)$ are null then the inequality
$$
1 \geq \P \left( \bigcup_{-n\leq\ell\leq n} A_{m,M}^{\delta,\ell} \right) = \sum_{-n\leq\ell\leq n} \P \big( A_{m,M}^{\delta,\ell} \big) = (2n+1) \P \big( A_{m,M}^\delta \big)
$$
leads to a contradiction as $n$ tends to infinity, since $\P\big(A_{m,M}^\delta\big)>0$ (Lemma \ref{2paths}). Consequently, $\P\big(A_{m,M}^{\delta,\ell}\cap A_{m,M}^{\delta,k}\big)>0$ for some $\ell,k\in\Z$.\\
On this event, there exist four infinite paths started in the cells $(0,2M\ell)+C_{m-\delta,M}$ and $(0,2Mk)+C_{m-\delta,M}$, leaving the cells $(0,2M\ell)+C_{m,M}$ and $(0,2Mk)+C_{m,M}$ through their east sides. We denote them by $\gamma_1,\gamma_2,\gamma_3,\gamma_4$ according to the ordinate at which they intersect the axis $\{x=m\}$. See Figure \ref{fig:3paths}. Now, among these four paths, at least three are disjoint. Indeed, the path $\gamma_1$ cannot touch $\gamma_3$, otherwise by planarity, it will necessarily touch $\gamma_2$ which is forbidden on the event $A_{m,M}^{\delta,\ell}\cap A_{m,M}^{\delta,k}$. As a consequence, $\gamma_1$ cannot touch $\gamma_4$.\\
It suffices to replace $M$ with $\max((2|\ell|+1)M,(2|k|+1)M)$ to conclude.
\end{proof}

\begin{figure}[!ht]
\begin{center}
\psfrag{x}{\small{$x=0$}}
\psfrag{1}{\small{$\gamma_1$}}
\psfrag{2}{\small{$\gamma_2$ and $\gamma_3$}}
\psfrag{4}{\small{$\gamma_4$}}
\includegraphics[width=10cm,height=8cm]{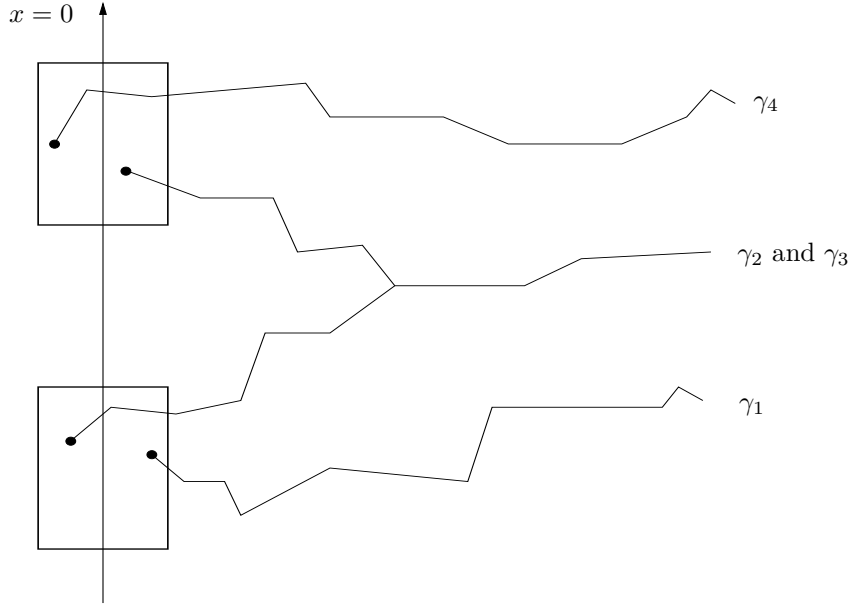}
\end{center}
\caption{\label{fig:3paths} {\small \textit{Here is a representation of the fourth paths $\gamma_1,\gamma_2,\gamma_3,\gamma_4$ corresponding to the event $A_{m,M}^{\ell}\cap A_{m,M}^{k}$. On the axis $x=0$, the two rectangles are the cells $(0,2M\ell)+C_{m,M}$ and $(0,2Mk)+C_{m,M}$. While the paths $\gamma_2$ and $\gamma_3$ coalesce, it remains three paths among $\gamma_1,\gamma_2,\gamma_3,\gamma_4$ which are disjoint.}}}
\end{figure}

The condition that the paths leave the cell $C_{m,M}$ through its east side is essential to obtain a contradiction in the previous proof. Indeed, one could imagine that the path $\gamma_2$ leaves the cell $(0,2M\ell)+C_{m,M}$ through its north side and goes into $(0,2Mk)+C_{m,M}$ (with $\ell<k$) through its south side in order to slip between $\gamma_3$ and $\gamma_4$. In this case, there is nothing to prevent the paths $\gamma_1$ and $\gamma_3$, respectively $\gamma_2$ and $\gamma_4$, from coalescing.\\

It is time to prove Lemma \ref{corol2}. The integers $m,M$ given by Lemma \ref{3paths} and for which $\P\big(B_{m,M}^{\delta}\big)>0$ will provide $\P(F_{m,M})>0$. Our idea is to modify the event $B_{m,M}^{\delta}$ by building a shield that will protect the cell $C_{m,M}$ and in particular the three disjoint paths given by $B^{\delta}_{m,M}$. \\
For that purpose and for the rest of this section, we denote by $B_{m,M}^{\delta}(\cdot)$ to specify which point process satisfies the event $B_{m,M}^{\delta}$.\\

Let $R$ be a real number larger than $2\max\{m,M\}$. Let us denote by $\Lambda_{m,M}^{R}$ the set
\begin{equation}
\Lambda_{m,M}^{R} = \Big( B(O,R)  \cap \{ x<m \} \Big) \setminus C_{m,M} ~.
\label{def:Lambda}\end{equation}
Now, let us put small balls in the set $\Lambda_{m,M}^{R}$ throughout the circle centered at the origin and with radius $R$. Let $\varepsilon$ be a small positive real number and $X_{0}=(-R+\varepsilon,0)$. We define a sequence $X_{0},X_{1},\ldots$ of points of $\R^2$ such that for any integer $k\geq 1$, $X_{k+1}$ has a positive ordinate and satisfies $|X_{k+1}-X_{k}|=1$ and $|X_{k+1}|=R-\varepsilon$. Let $n=n(m,M,R,\varepsilon)$ be the smallest index $k$ such that $X_{k}$ belongs to the set $\{x\geq-m\}$. Actually:
$$
X_k=(R-\varepsilon)e^{\i k\alpha},\quad k\leq n(m,M,R,\varepsilon)\;\mbox{ and }\;\alpha=2\mbox{arcsin}\Big(\frac{1}{2(R-\varepsilon)}\Big) ~.
$$
Let us add a last point $X_{n+1}$ with positive ordinate, abscissa equal to $m-\delta/2-\varepsilon$ and such that $|X_{n+1}|=R-\varepsilon$. By symmetry with respect to the axis $y=0$, we define the sequence $X_{-1},\ldots,X_{-n},X_{-n-1}$. The balls centered at the $X_k$'s, $-n-1\leq k\leq n+1$, with radius $\varepsilon$ are all included in $\Lambda_{m,M}^{R}$ and form a shield protecting the cell $C_{m,M}$. See Figure \ref{fig:shield}.\\

\begin{figure}[!ht]
\begin{center}
\psfrag{1}{\small{$\gamma_Y$}}
\psfrag{2}{\small{$\gamma_X$}}
\psfrag{3}{\small{$\gamma_Z$}}
\includegraphics[width=13cm,height=11cm]{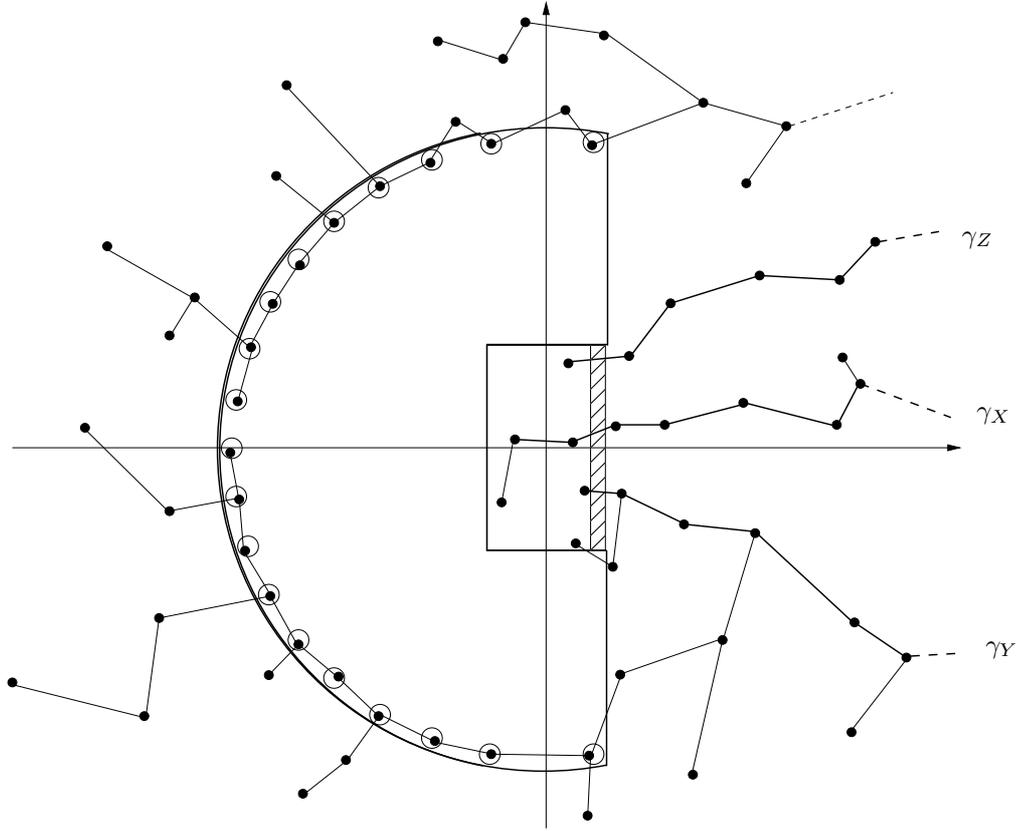}
\end{center}
\caption{\label{fig:shield} {\small \textit{Here are the set $\Lambda^R_{m,M}$ and the Directed Spanning Forest constructed on the PPP $N$ satisfying the event $D_{R,\varepsilon}\cap B_{m,M}^{\delta}$. The disjoint infinite paths $\gamma_X,\ \gamma_Y,\ \gamma_Z$ of the event $B^\delta_{m,M}$ are represented in bold. The points $X,\ Y,\ Z\in C_{m-\delta,M}$ are in gray. The hatched area is the strip of width $\delta$ between $E_{m-\delta,M}$ and $E_{m,M}$ that the paths $\gamma_X,\ \gamma_Y$ and $\gamma_Z$ have to cross.}}}
\end{figure}

\noindent
Let $D_{R,\varepsilon}(N)$ be the event such that each ball $B(X_k,\varepsilon)$, for $-n-1\leq k\leq n+1$, contains exactly one point of the PPP $N$, and the rest of $\Lambda_{m,M}^{R}$ is empty. This event occurs with positive probability for any $R,\varepsilon$. In the sequel, we will denote by $Y_k$ the point contained in the ball $B(X_k,\varepsilon)$, $-n-1\leq k\leq n+1$.\\
Let $N^{in}$ and $N^{out}$ be two independent PPP with respective intensities 1 on $\Lambda^R_{m,M}$ and on $\R^2\setminus \Lambda^R_{m,M}$, then their superposition $N^{in}+N^{out}$ is a PPP of intensity 1 on the whole plane. Hence,
\begin{equation}
\label{P1>0}
\P \big( D_{R,\varepsilon}(N^{in}) \big) = \P \big( D_{R,\varepsilon}(N^{in}+N^{out}) \big) = \P \big( D_{R,\varepsilon}(N) \big)  > 0 ~.
\end{equation}
The first equality in \eqref{P1>0} comes from the fact that the realization of $D_{R,\varepsilon}(N^{in}+N^{out})$ depends only on points in the set $\Lambda^R_{m,M}$, and the second equality is due to the fact that $N$ and $N^{in}+N^{out}$ have the same distribution.\\
Since deleting the points of $N^{in}$ does not affect the existence of the three paths $\gamma_X$, $\gamma_Y$ and $\gamma_Z$ of $B^\delta_{m,M}(N^{in}+N^{out})$ when this event is realized, we have $B_{m,M}^{\delta}(N^{in}+N^{out})\subset B_{m,M}^{\delta}(N^{out})$ and
\begin{equation}
\label{P2>0}
\P \big( B_{m,M}^{\delta}(N^{out}) \big) \geq \P \big( B_{m,M}^{\delta}(N^{in}+N^{out}) \big)=\P \big( B_{m,M}^{\delta}(N) \big)  > 0 ~.
\end{equation}
Moreover, if $B_{m,M}^{\delta}(N^{out})$ is satisfied, so is $B_{m,M}^{\delta}(N^{in}+N^{out})$ provided the points of $N^{in}\cap\{-m\leq x< m\}$ do not modify the DSF in the cell $C_{m,M}$. This is false in general. This becomes true as soon as $D_{R,\varepsilon}(N^{in})$ is satisfied with $R>2\sqrt{(2m)^2+(2M)^2}+1$ and $\varepsilon<1/2$, thanks to the following result.

\begin{lemme}
\label{lengthedge}
Let $m,M$ be positive integers and assume that at least two edges in the DSF start in the cell $C_{m,M}$ and exit it through its east side. Then, any edge whose west vertex belongs to the cell $C_{m,M}$ has a length smaller than $2\sqrt{(2m)^2+(2M)^2}$.
\end{lemme}

\noindent
Lemma \ref{lengthedge} will be proved at the end of this section. To sum up, for $R$ and $\varepsilon$ as above;
\begin{equation}
\label{intersection}
D_{R,\varepsilon}(N^{in}) \cap B_{m,M}^{\delta}(N^{out}) = D_{R,\varepsilon}(N^{in}+N^{out}) \cap B_{m,M}^{\delta}(N^{in}+N^{out}) ~.
\end{equation}
Combining with (\ref{P1>0}), (\ref{P2>0}) and (\ref{intersection}) it follows:
\begin{eqnarray*}
\P \big( D_{R,\varepsilon}(N) \cap B_{m,M}^{\delta}(N) \big)
& = & \P \big(   D_{R,\varepsilon}(N^{in}+N^{out}) \cap B_{m,M}^{\delta}(N^{in}+N^{out}) \big) \\
& = & \P \big( D_{R,\varepsilon}(N^{in}) \cap  B_{m,M}^{\delta}(N^{out}) \big) \\
& = & \P \big( D_{R,\varepsilon}(N^{in}) \big) \P \big(  B_{m,M}^{\delta}(N^{out}) \big)  \; > \; 0 ~.
\end{eqnarray*}
To conclude the proof, it remains to prove the inclusion
\begin{equation}
\label{inclusionfinale}
D_{R,\varepsilon} \cap B_{m,M}^{\delta} \subset F_{m,M} ~.
\end{equation}
Let $\gamma_X,\gamma_Y,\gamma_Z$ be the three disjoint infinite paths given by $B_{m,M}^{\delta}$. Let us denote by $\gamma_X$ the path which intersects the axis $\{x=m\}$ at the intermediate ordinate. On the one hand, we choose $\varepsilon>0$ small enough so that the abscissa of $Y_{k}$, $0\leq k\leq n$, is smaller than that of $Y_{k+1}$. Hence, $Y_{k}$ will prefer to have $Y_{k+1}$ as ancestor rather than a point of $N\cap C_{m,M}$. Similarly $Y_{-k}$ will prefer to have $Y_{-k-1}$ as ancestor. On the other hand, since the abscissa of $Y_{n+1}$ is larger than $m-\delta$ (with $\varepsilon<\delta/4$), $Y_{n+1}$ cannot have a point of $\gamma_X$ as ancestor. Otherwise, the path of $Y_{n+1}$ would cross $\gamma_Y$ or $\gamma_Z$ which is impossible. Similarly, $Y_{-n-1}$ cannot have a point of $\gamma_X$ as ancestor. Henceforth, assume the event $D_{R,\varepsilon}\cap B_{m,M}^{\delta}$ satisfied. The previous construction prevents any path $\gamma_A$ in the DSF with $A\in N \cap \{x<m\}\setminus C_{m,M}$ from touching $\gamma_X$ on $\{x\leq m\}$. It can neither coalesce with $\gamma_X$ on $\{x>m\}$ since on the half-plane $\{x>m-\delta\}$, $\gamma_X$ is trapped between $\gamma_Y$ and $\gamma_Z$. This leads to (\ref{inclusionfinale}) and concludes the proof of Lemma \ref{corol2}.\\

The section ends with the proof of Lemma \ref{lengthedge}.

\begin{proof}
First remark that an edge whose both vertices belong to the cell $C_{m,M}$ has necessarily a length smaller than $\sqrt{(2m)^2+(2M)^2}$. Now, let us focus on an edge $\{X,Y\}$ in the DSF such that $X\in C_{m,M}$ and $Y\notin C_{m,M}$. By hypothesis, there exists another edge $\{X',Y'\}$ such that $X'\in C_{m,M}$ and $Y'\in\{x\geq m\}$. If the abscissa of $X$ is smaller than that of $X'$ then $X'$ could be the ancestor of $X$. So,
$$
|X-Y| \leq |X-X'| \leq \sqrt{(2m)^2+(2M)^2} ~.
$$
Otherwise, the same argument leads to $|X'-Y'| \leq \sqrt{(2m)^2+(2M)^2}$. Moreover, $Y'$ could be an ancestor of $X$. So,
$$
|X-Y| \leq |X-Y'| \leq |X-X'| + |X'-Y'| \leq 2 \sqrt{(2m)^2+(2M)^2} ~.
$$
This concludes the proof.\end{proof}

\section{DSF on a PPP with random holes}
\label{sectiontrous}

Let us now consider a Boolean model $\Gamma$ where the germs are distributed following a PPP $Q$ with intensity $\lambda>0$, independent of the PPP $N$, and where the grains are balls with fixed radius $r>0$:
\begin{equation}
\Gamma=\bigcup_{X\in Q} B(X,r) \label{modeleBooleen}
\end{equation}
(see \eg \cite{molchanov,schneiderweil} for an exposition). We delete the points of $N$ that belong to $\Gamma$. This amounts to considering the DSF with direction $e_x$ and whose vertices are the points of $N\cap \Gamma^c$ where $\Gamma^c=\R^2\setminus \Gamma$. Intuitively, the holes created by $\Gamma$ act like obstacles, which are difficult to cross and may prevent the coalescence of different paths. Is such a DSF still a tree ? The answer is still ``yes''.

\begin{figure}[!ht]
\begin{center}
\includegraphics[width=6cm,height=6cm]{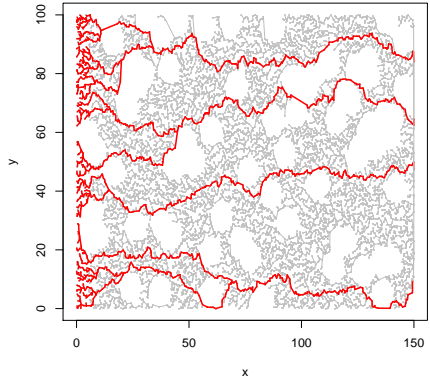}
\caption{{\small \textit{Simulations of the Directed Spanning Forest with Boolean holes. The paths with direction $e_{x}$ and coming from vertices with abscissa $0\leq x\leq 5$ and their ancestors are represented in bold red lines. We see on (a) that the Boolean holes mostly deviate the paths of the DSF, as in abscissa 50 ordinate 55. However, when a path is trapped, as in abscissa 100 ordinate 70, it has to cross the hole.}}}\label{fig2}
\end{center}
\end{figure}

\begin{theorem}
\label{th1boolean}
The DSF constructed on $N\cap \Gamma^c$ is almost surely a tree.
\end{theorem}

Let us remark that Theorem \ref{th1boolean} holds whatever the value of the intensity $\lambda$ of the Boolean model. In particular, the DSF is a tree even if $\Gamma$ contains unbounded components. See \cite{MR} for a complete reference on continuum percolation.\\
Moreover, Theorem \ref{th1boolean} remains true when the radii are i.i.d. random variables, provided the support of their common distribution is bounded.

\begin{proof}
We generalize the proof of Theorem \ref{th1}. Actually, thanks to the translation invariance of the joint process $(N,Q)$, most of arguments used in Section \ref{proof} still hold for the DSF constructed on $N\cap \Gamma^c$. Hence, we can assume that for any $\delta>0$ there exist $m,M\in \N^*$ such that the event $B^\delta_{m,M}(N,Q)$ has a positive probability (i.e. Lemma \ref{3paths}). It remains to prove that $\P(F_{m,M}(N,Q))>0$.\\
The main difference with the proof of Theorem \ref{th1} lies in the construction of the shield (Fig. \ref{fig:shield}). Indeed, the Boolean model $\Gamma$ could create holes in the shield so that the inclusion $D_{R,\varepsilon}\cap B_{m,M}^{\delta}\subset F_{m,M}$ is no longer true. To avoid this difficulty, it suffices to require that $\Gamma$ does not intersect any ball $B(X_k,\varepsilon)$ involved in the shield. For that purpose, let us introduce the event $H_{R,\varepsilon}(Q)$ defined as follows: there is no points of $Q$ in
$$
\Upsilon_{R,\varepsilon} = \big(B(O,R+r)\setminus B(O,R-r-\varepsilon)\big)\cap \{x<m\}
$$
(it can be assumed without restriction that $R>r+\varepsilon$). \\
Let $\delta=2r$. Henceforth, on the event $H_{R,\varepsilon}(Q)$ the Boolean model $\Gamma$ cannot intersect any ball $B(X_k,\varepsilon)$ with $|k|\leq n+1$. The same shield as in the previous section can be constructed in order to protect the three disjoint paths given by $B^\delta_{m,M}(N,Q)$. This leads to
\begin{equation}
\label{inclusionfinale2}
B^\delta_{m,M}(N,Q) \cap D_{R,\varepsilon}(N) \cap H_{R,\varepsilon}(Q) \subset F_{m,M}
\end{equation}
The arguments to prove that $B^\delta_{m,M}(N,Q)\cap D_{R,\varepsilon}(N)\cap H_{R,\varepsilon}(Q)$ occurs with positive probability are the same as in Section \ref{section:P(F)>0}. Let us consider four independent PPP: $N^{in}$ and $N^{out}$ with intensity $1$ on $\Lambda^R_{m,M}$ and $\R^2\setminus \Lambda^R_{m,M}$, and $Q^{in}$ and $Q^{out}$ with intensity $\lambda$ on $\Upsilon_{R,\varepsilon}$ and $\R^2\setminus \Upsilon_{R,\varepsilon}$. The existence of the three paths $\gamma_X$, $\gamma_Y$ and $\gamma_Z$ of $B^{\delta}_{m,M}(N,Q)$ is not affected if we change the point process $N$ in $\Lambda^R_{m,M}$ and if we delete the points of $Q$ in $\Upsilon_{R,\varepsilon}$:
$$
B^\delta_{m,M}(N^{in}+N^{out},Q^{in}+Q^{out}) \subset B^\delta_{m,M}(N^{out},Q^{out}) ~.
$$
This inclusion becomes an equality whenever the events $D_{R,\varepsilon}(N^{in})$ and $ H_{R,\varepsilon}(Q^{in})$ are satisfied:
\begin{multline}
\label{intersection2}
B^\delta_{m,M}(N^{in}+N^{out},Q^{in}+Q^{out}) \cap D_{R,\varepsilon}(N^{in}+N^{out}) \cap H_{R,\varepsilon}(Q^{in}+Q^{out})\\
= B^\delta_{m,M}(N^{out},Q^{out}) \cap D_{R,\varepsilon}(N^{in}) \cap H_{R,\varepsilon}(Q^{in})
\end{multline}
(as in \eqref{intersection}). Relations (\ref{inclusionfinale2}) and (\ref{intersection2}) imply:
\begin{align}
\label{etape3}
\P(F_{m,M})\geq & \P\big(B^\delta_{m,M}(N,Q)\cap D_{R,\varepsilon}(N)\cap H_{R,\varepsilon}(Q)\big)\nonumber\\
= & \P\big(B^\delta_{m,M}(N^{in}+N^{out},Q^{in}+Q^{out})\cap D_{R,\varepsilon}(N^{in}+N^{out})\cap H_{R,\varepsilon}(Q^{in}+Q^{out})\big)\nonumber\\
= & \P\big(B^\delta_{m,M}(N^{out},Q^{out})\cap D_{R,\varepsilon}(N^{in})\cap H_{R,\varepsilon}(Q^{in})\big)\nonumber\\
= & \P\big(B^\delta_{m,M}(N^{out},Q^{out})\big)\P\big(D_{R,\varepsilon}(N^{in})\big)\P\big(H_{R,\varepsilon}(Q^{in})\big) ~.
\end{align}
The proof ends with $\P(D_{R,\varepsilon}(N^{in}))>0$,
$$
\P \big( B^\delta_{m,M}(N^{out},Q^{out}) \big) \geq \P \big( B^\delta_{m,M}(N^{in}+N^{out},Q^{in}+Q^{out}) \big) = \P \big( B^\delta_{m,M}(N,Q) \big) > 0
$$
and
$$
\P \big( H_{R,\varepsilon}(Q^{in}) \big) = \P \big( Q(\Upsilon_{R,\varepsilon}) = 0 \big) \geq e^{-\lambda\pi((R+r)^2-(R-r-\varepsilon)^2)} > 0 ~.
$$
\end{proof}

\section{There is no bi-infinite path in the DSF}
\label{sectionbiinfini}

By construction of the DSF with direction $e_{x}$, each path is infinite to the right \ie with direction $e_{x}$. A path $\gamma$ of the DSF is said to be \textit{bi-infinite} if it is also infinite to the left \ie with direction $-e_{x}$. In other words, every point $X\in N$ of a bi-infinite path $\gamma$ is the ancestor of another point of $N$ (which belongs to $\gamma$ too).\\
As an illustration of Theorem \ref{th1}, we use the fact that the DSF is a tree to prove:

\begin{theorem}
\label{thbiinf}
There is a.s. no bi-infinite path in the DSF.
\end{theorem}

\begin{proof}
Considering the DSF as a subset of $\R^{2}$, we can define the number $R_I$ of intersection points of bi-infinite paths with a given vertical interval $I$ of $\R^2$:
\begin{equation}
R_I = \card\{ (x,y) \in I ;\  (x,y) \mbox{ belongs to a bi-infinite path} \} ~.
\end{equation}
Every bi-infinite path a.s. crosses any given vertical axis. Hence, the announced result becomes:
\begin{equation}
\label{butbiinf}
\P \big( R_{\{0\}\times \R} = 0 \big) = 1 ~.
\end{equation}
For a point $X$ of the DSF (not necessarily a point of $N$), let us denote by $\mathcal{T}_X$ the subtree of the DSF consisting of the bi-infinite paths going through $X$.\\
Let $I$ and $J$ be two vertical intervals of $\R^2$ such that the abscissa of $I$ is smaller than the one of $J$. We define by $\widetilde{R}_{I,J}$ the number of intersection points $X$ of bi-infinite paths with $J$ whose subtrees $\mathcal{T}_{X}$ intersect $I$:
\begin{equation}
\widetilde{R}_{I,J}=\card\{X\in J ;  X\mbox{ belongs to the DSF and }\mathcal{T}_{X}\cap I\not=\emptyset\}.
\end{equation}
Remark that the random variables $R_I$ and $\widetilde{R}_{I,J}$ are integrable as soon as $I$ has finite Lebesgue measure (they are bounded by the number of edges of the DSF crossing this interval).\\

The idea is to state with stationary arguments that for every $L>0$:
\begin{equation}
\label{but2}
\E\big(R_{\{0\}\times [0,L]}\big)=\E\big(\widetilde{R}_{\{0\}\times [0,L], \{x\}\times \R}\big).
\end{equation}
If \eqref{but2} holds then the coalescence of infinite paths (Theorem \ref{th1}) forces both members of the equality to be zero. Indeed, as a decreasing sequence of nonnegative integers, the limit of $(\widetilde{R}_{\{0\}\times[0,L], \{x\}\times \R})_{x>0}$, as $x\to+\infty$, exists. It is smaller than $1$ by Theorem \ref{th1}. Thanks to \eqref{but2} and stationarity we get for every $L>0$:
\begin{align}
\label{but3}
1\geq \E\big(R_{\{0\}\times [0,L]}\big)=L \E\big(R_{\{0\}\times [0,1]}\big).
\end{align}
Letting $L\rightarrow +\infty$ in \eqref{but3}, the expectation of $R_{\{0\}\times [0,1]}$ is necessarily $0$ and \eqref{butbiinf} follows.\\

Let us now prove \eqref{but2}. For $x\in \R$ and $L>0$, all bi-infinite paths crossing $\{0\}\times [0,L]$ also cross the vertical line $\{x\}\times \R$. Thus, $\widetilde{R}_{\{0\}\times [0,L],\{x\}\times \R}$ is smaller than $R_{\{0\}\times [0,L]}$. So do their expectations:
\begin{equation}
\label{etape8}
\E\big(\widetilde{R}_{\{0\}\times [0,L],\{x\}\times \R}\big)\leq \E\big(R_{\{0\}\times [0,L]}\big) .
\end{equation}
To show that (\ref{etape8}) is actually an equality, let us first remark that $R_{\{0\}\times [0,L]}$ and $R_{\{x\}\times [0,L]}$ have the same expectation. The bi-infinite paths crossing $\{x\}\times [0,L]$ can be traced back to intersection points of bi-infinite paths with segments of the form $\{0\}\times [nL,(n+1)L]$, $n\in \Z$:
\begin{equation}
R_{\{x\}\times [0,L]}= \widetilde{R}_{\{0\}\times \R,\{x\}\times [0,L]}\leq \sum_{n\in \Z}\widetilde{R}_{\{0\}\times [nL,(n+1)L],\{x\}\times [0,L]}.\label{etape9}
\end{equation}
Thanks to stationarity by vertical translation of $-nL$ for all $n\in \Z$, $\widetilde{R}_{\{0\}\times [nL,(n+1)L],\{x\}\times [0,L]}$ and $\widetilde{R}_{\{0\}\times [0,L],\{x\}\times [-nL,-(n-1)L]}$ are identically distributed. Thus, \eqref{etape9} gives:
\begin{equation*}
\E\big(R_{\{0\}\times [0,L]}\big)=\E\big(R_{\{x\}\times [0,L]}\big)\leq \E\Big(\sum_{n\in \Z}\widetilde{R}_{\{0\}\times [0,L],\{x\}\times [-nL,-(n-1)L]}\Big)=\E\Big(\widetilde{R}_{\{0\}\times [0,L],\{x\}\times \R}\Big).
\end{equation*}
This provides \eqref{but2} and proves Theorem \ref{thbiinf}.
\end{proof}

Finally, let us notice that an alternative proof exists based on the notion of bifurcating point. See the proof of Theorem 2.2 of \cite{gangopadhyayroysarkar}.

\section{A final remark}

Let us describe an elementary model of coalescing random walks on $\Z^{2}_{\mathcal{E}}=\{(i,j)\in\Z^{2},i+j \mbox{ is even}\}$. Each vertex $(i,j)$ of $\Z^{2}_{\mathcal{E}}$ is connected to $(i+1,j+1)$ or $(i+1,j-1)$ with probability $\frac{1}{2}$, and independently from the others. This process provides a forest $\mathcal{F}$ with direction $e_{x}$ spanning all $\Z^{2}_{\mathcal{E}}$. Using classical results on simple random walks, one can prove that $\mathcal{F}$ is a.s. a tree (\eg \cite{fontesisopinewmanravishankar}).\\
Now, let us define a graph $\mathcal{F}^{\ast}$ on $\Z^{2}_{\mathcal{O}}=\{(i,j)\in\Z^{2},i+j \mbox{ is odd}\}$ as follows: $(i,j)\in\Z^{2}_{\mathcal{O}}$ is connected to $(i-1,j+1)$ (resp. to $(i-1,j-1)$) if and only if the vertex $(i-1,j)\in\Z^{2}_{\mathcal{E}}$ is connected to $(i,j-1)$ (resp. to $(i,j+1)$) in $\mathcal{F}$. First, the forests $\mathcal{F}$ and $\mathcal{F}^{\ast}$ are identically distributed, which implies that $\mathcal{F}^{\ast}$ is also a tree. Moreover, $\mathcal{F}^{\ast}$ is the dual graph of $\mathcal{F}$: since $\mathcal{F}^{\ast}$ is a tree, there is no bi-infinite path in $\mathcal{F}$.\\
In the same way, we wonder if the DSF constructed on the PPP $N$ admits a dual forest with the same distribution, which will allow to derive Theorem \ref{thbiinf} from applying Theorem \ref{th1} to this dual forest.

\bigskip
\noindent \textbf{Acknowledgements: }The authors are grateful to Fran\c{c}ois Baccelli for introducing the DSF model to them. The authors also thank the members of the "Groupe de travail G\'{e}om\'{e}trie Stochastique" of Universit\'{e} Lille 1 for enriching discussions.\\

{\small \providecommand{\noopsort}[1]{}

}


\begin{thebibliography}{10}

\bibitem{alexander}
K.S. Alexander.
\newblock Percolation and minimal spanning forest in infinite graphs.
\newblock {\em The Annals of Probability}, 23(1):87--104, 1995.

\bibitem{athreyaroysarkar}
S.~Athreya, R.~Roy, and A.~Sarkar.
\newblock Random directed trees and forest - drainage networks with dependence.
\newblock {\em Electronic journal of Probability}, 13:2160--2189, 2008.
\newblock Paper no.71.

\bibitem{baccellibordenave}
F.~Bacelli and C.~Bordenave.
\newblock The radial spanning tree of a {P}oisson point process.
\newblock {\em Annals of Applied Probability}, 17(1):305--359, 2007.

\bibitem{bonichonmarckert}
N.~Bonichon and J.-F. Marckert.
\newblock Asymptotic of geometrical navigation on a random set of points of the
  plane.
\newblock 2010.
\newblock Submitted.

\bibitem{burtonkeane}
R.M. Burton and M.S. Keane.
\newblock Density and uniqueness in percolation.
\newblock {\em Comm. Math. Phys.}, 121:501--505, 1989.

\bibitem{ferrarilandimthorisson}
P.~A. Ferrari, C.~Landim, and H.~Thorisson.
\newblock Poisson trees, succession lines and coalescing random walks.
\newblock {\em Annales de l'Institut Henri Poincar\'{e}}, 40:141--152, 2004.
\newblock Probabilit\'{e}s et Statistiques.

\bibitem{FP}
P.~A. Ferrari and L.~P.~R. Pimentel.
\newblock Competition interfaces and second class particles.
\newblock {\em Ann. Probab.}, 33(4):1235--1254, 2005.

\bibitem{fontesisopinewmanravishankar}
L.~R.~G. Fontes, M.~Isopi, C.~M. Newman, and K.~Ravishankar.
\newblock The brownian web: characterization and convergence.
\newblock {\em Ann. Probab.}, 32(4):2857--2883, 2004.

\bibitem{gangopadhyayroysarkar}
S.~Gangopadhyay, R.~Roy, and A.~Sarkar.
\newblock Random oriented trees: a model of drainage networks.
\newblock {\em Ann. App. Probab.}, 14(3):1242--1266, 2004.

\bibitem{howardnewman2}
C.~D. Howard and C.~M. Newman.
\newblock Euclidean models of first-passage percolation.
\newblock {\em Probab. Theory Related Fields}, 108(2):153--170, 1997.

\bibitem{liceanewman}
C.~Licea and C.~M. Newman.
\newblock Geodesics in two-dimensional first-passage percolation.
\newblock {\em Ann. Probab.}, 24(1):399--410, 1996.

\bibitem{MR}
R.~Meester and R.~Roy.
\newblock {\em Continuum percolation}, volume 119 of {\em Cambridge Tracts in
  Mathematics}.
\newblock Cambridge University Press, Cambridge, 1996.

\bibitem{molchanov}
I.~Molchanov.
\newblock {\em Statistics of the {B}oolean model for practitioners and
  mathematicians}.
\newblock Chichester, wiley edition, 1997.

\bibitem{schneiderweil}
R.~Schneider and W.~Weil.
\newblock {\em Stochatic and integral geometry}.
\newblock Probability and its Applications. New York, springer-verlag edition,
  2008.

\end{thebibliography}
\end{document}